\documentclass[12pt]{amsart}
\usepackage{amssymb}
\usepackage{hyperref}   
\hypersetup{
    colorlinks=true,
    linkcolor=black,
    citecolor=black,
    filecolor=black,
   urlcolor=black}

\newtheorem{thm}{Theorem}
\newtheorem{prop}{Proposition}
\newtheorem*{thmoneprime}{Theorem 1'}
\newtheorem*{thmtwoprime}{Theorem 3'}
\newtheorem{lem}{Lemma}

\theoremstyle{definition}
\newtheorem*{defn}{Definition}

\def\bN{\mathbb{N}}

\def\al{\alpha}
\def\be{\beta}
\def\ga{\gamma}

\def\eps{\epsilon}
\def\om{\omega}
\def\Om{\Omega}
\newcommand{\m}[1]{\mathcal{#1}}

\usepackage[ marginparsep=5mm,heightrounded,centering,margin=3.47cm]{geometry}

\title[On Growth of random groups of intermediate growth]{On Growth of random groups \\ of intermediate growth}

\author{Mustafa G. Benli, Rostislav Grigorchuk and Yaroslav Vorobets}

\address{Mustafa G. Benli }
\email{mbenli@math.tamu.edu}

\address{Rostislav Grigorchuk }
\email{grigorch@math.tamu.edu}

\address{Yaroslav Vorobets}
\email{yvorobet@math.tamu.edu}

\address{Texas A{\&}M University \\ Mailstop 3368 \\ College Station \\
TX 77843--3368 \\ USA.}

\date{April 28 , 2013\\
The first two authors were supported by NSF grant DMS-1207699}

\begin{document}

  \begin{abstract} 
  
  We study the growth of typical groups from the family of $p$-groups of intermediate growth
  constructed by the second author. We find that, in the sense of category, a generic group exhibits 
  oscillating growth with no universal upper bound.  At the same time, from a measure-theoretic point of view (i.e., almost surely relative to an appropriately
  chosen probability measure), the growth function is bounded by $e^{n^\alpha}$ for some $\alpha<1$.

%

 \end{abstract}

 \maketitle

 \section{Introduction}

There are  few approaches to randomness in group theory. The most known are associated with the names of
Gromov and Olshanskii. For the account of these approaches and further literature see,  \cite{Oliv}. As far as
the authors are concerned, these approaches deal with the models    when in certain classes of  finitely
presented groups  one locates "generic"  finite presentations with prescribed properties.   These models
are variation of the \emph{density model} and randomness in them appears in the form of \emph{frequency} or
\emph{density}, which correspond to the classic "naive" approach to probability in mathematics.

The modern Kolmogorov's approach to probability assumes   existence of  a space supplied with a sigma-algebra
of measurable sets and a probability measure on it.

There can be different constructions of spaces of groups and one of them was suggested in
\cite{grigorch:degrees}, where
 Gromov's idea from \cite{gromov:poly_growth} of convergence of marked metric spaces was transformed into the notion
 of the compact totally
disconnected topological metrizable  space $\mathcal M_k$ of marked $k$-generated groups, $k \geq 2$. Later
it was discovered that  this topology is related to the Chabauty topology in the space of normal subgroups of the
free group of rank $k$ \cite{champ:grps_fini}. Observe that in general the Chabauty topology is defined in the
space of closed subgroups of a locally compact group (there is analogous notion in differential geometry
\cite{benedetti:book}), and that in  the case of a discrete group $G$ it is nothing but the topology induced
on the set of subgroups of $G$ by the Tychonoff topology of the space $\{0,1\}^G$. 

The main result of \cite{grigorch:degrees}  is the  construction of the first examples of groups of
intermediate growth, thus answering a question of Milnor \cite{milnor:note68}. In fact, an uncountable
family of 3-generated groups was introduced and studied in \cite{grigorch:degrees}, and among other results
it was shown  there that the  set of rates of growth of finitely generated groups has the cardinality of the
continuum, and that there are pairs of groups with incomparable growth (the growth rates of two groups are different
but neither grows faster than the other; in fact, the space of rates of growth of
finitely generated groups contains an anti-chain of the cardinality of the continuum).
The possibility of
such phenomenon is based on the fact that there are groups with  \emph{oscillating growth}, i.e. groups whose growth on
different parts of the range of the argument of growth function (which is a set of natural numbers)
behaves alternatively in two fashions: in the intermediate (between polynomial and exponential) way and
exponential way. In this paper we will use one  particular form of oscillating property which will be
defined below.

The construction in \cite{grigorch:degrees} deals  with torsion 2-groups of intermediate growth. 
It also provides interesting examples of self-similar groups and 
first examples of just-infinite branch groups (see \cite{grigorchuk-s:standrews}).
Later a similar
construction of $p$-groups of intermediate growth was produced for arbitrary prime $p$ as well as the first
example of a torsion free group of intermediate growth \cite{grigorch:degrees85}.

The reason for introducing   the space $\mathcal M_k$  of marked groups in \cite{grigorch:degrees}   was to
show that this space in the cases  $k=2,3$ (and hence for all $k$) contains a closed subset of groups 
homeomorphic to a Cantor set consisting primarily of torsion groups of intermediate growth.  
Later other interesting families of
groups constituting a Cantor set of groups and satisfying various properties were produced and used for
answering different questions  \cite{champ:grps_fini,nekrash:minimal07}.  The topology in the space
$\mathcal M_k$ was used in \cite{grigorch:degrees,grigorch:habil} not only for study of growth but also for
investigations of algebraic properties of the involved groups. For instance, among  many ways of showing that the
involved groups are not finitely presented,
 there is one which makes use of this topology
(the topic of finite presentability  of groups in the context of growth, amenability and topology is
discussed in detail in \cite{BGH12}). At present there is a big account of results related to the 
space of marked groups
and various algebraic, geometric and asymptotic properties of groups including such properties as
(T)-property of Kazhdan, local embeddability into finite or amenable groups (so-called LEF and LEA properties), 
being sofic and various other properties (see \cite{tullio:book} for a
comprehensive source of these).

For each $k \geq 2$ there is a natural embedding of $\mathcal M_k$ into $\mathcal M_{k+1}$  and one can
consider the inductive limit $\mathcal M= \lim_k \mathcal M_k$ which is a locally compact totally
disconnected space.
As  was observed by Champetier  in \cite{champ:grps_fini} the group of Nielsen
transformations over  infinite generating set acts naturally on this space with  orbits consisting of
isomorphic groups. Any Baire  measure on $\m{M}$, i.e., a measure defined on the sigma algebra generated by compact $G_\delta$
sets (countable intersections of open sets) with
finite values on compact sets, that is invariant (or at least quasi-invariant) with
respect to this action would be a good choice for the model of random finitely generated group  (this
approach based on discussions of the second author with E. Ghys is presented in \cite{grigorch:solved}).
Unfortunately, at the moment no such measures were produced. This is also related to the question of
existence of "good" measures  invariant (or quasi-invariant) under the action of the 
automorphism group of a free group  $F_k$ of rank $k\geq 2$, 
with support in the set of normal subgroups of $F_k$.

Fortunately, another approach can be used. It is based on the following idea.  Assume we have a compact
$X\subset \mathcal M_k$ of groups and a continuous map $\tau: X\rightarrow X$. Then by the Bogolyubov-Krylov
Theorem there is at least one $\tau$-invariant probability measure $\mu$ on $X$.  Suppose also that we have a
certain  group property $\mathcal{P}$ (or a family of properties), 
and that the subset $X_{\mathcal{P}}\subset X$ of groups
satisfying this property is measurable $\tau$-invariant (i.e. $\tau^{-1}(X_{\mathcal{P}})=X_{\mathcal{P}}$).  
Then one may be interested in the
measure $\mu(X_{\mathcal{P}})$ which is 0 or 1 in the case of ergodic measure  
(i.e.  when the only invariant measurable subsets up
to sets of $\mu$ measure 0 are empty set and the whole set $X$).  Observe that 
 by (another) Bogolyubov-Krylov theorem, ergodic measure  always
exist in the situation of a continuous map on a  metrizable compact space and are just extreme
points of the simplex of invariant measures. The described model allows to 
speak about typical properties of a random group from the family $(X,\mu)$.

The alternative approach when the measure $\mu$ is not specified is the study of the typical properties of groups
in compact $X$ from topological (or categorical) point of view i.e., in the sense of Baire category.  Under
this approach a group property $\mathcal{P}$ is typical if the subset $X_\mathcal{P}$ is co-meager i.e. its complement
$X\setminus X_\mathcal{P}$ is meager (a countable union of nowhere dense subsets of $X$).  Its happens quite often that
what is typical in the measure sense is not typical in the sense of category and this paper  gives one more
example of this sort.

\section{Statement of main results}
\label{sec2}

In Ergodic Theory  (and more generally in Probability Theory), one of the most important models is the model
of a shift in a space of sequences.  Given a finite alphabet $Y=\{s_1, \dots, s_k\}$,  one considers a space
$\Omega=Y^{\mathbb N}$ of infinite sequences $\om=(\om_n)_{n=1}^{\infty}, \om_n \in Y$ endowed with the
Tychonoff product topology. A natural transformation in such space is a shift  $\tau: \Omega\rightarrow
\Omega, (\tau(\om))_{n}=\om_{n+1}$. There is a lot of invariant  measures for the dynamical  system
$(\Omega,\tau)$ and in fact the simplex $M_{\tau}(\Omega)$ of invariant measures is Poulsen simplex (i.e.
ergodic measures are dense in weak-$\ast$ topology).

Let $p$ be  prime  and consider the set $\{0,1,\dots, p\}$ as an alphabet with the corresponding set
$\Omega_p=\{0,1,\ldots,p\}^\mathbb{N}$ of infinite sequences endowed with the shift $\tau:\Om_p\rightarrow \Om_p$. 
Let $\Om_{p,0}$ denote the subset of sequences which are eventually constant and
$\Om_{p,\infty}$ the set of sequences in which all symbols $\{0,1,\ldots,p\}$ appear infinitely
often. Note that  $\Om_{p,\infty}$ and $\Om_{p,0}$ are $\tau$ invariant.

In  \cite{grigorch:degrees,grigorch:degrees85}   for each $\om \in \Omega_p$ a group  $G_{\om}$  with a set
$S_{\om}=\{a,b_{\om}, c_{\om}\}$ of three generators acting on the  interval $[0,1]$  by  Lebesgue measure
preserving  transformations was constructed.
One of the specific features of this
construction is  that if two sequences $\om,\eta \in \Omega_p$, which are not eventually constant,
have the same prefix of length $n$ then the
corresponding groups  $G_\om,G_\eta$  have isomorphic Cayley graphs in the neighborhood of the identity
element of radius $2^{n-1}$.  Replacing the groups 
$G_\om , \om \in  \Om_{p,0}$ with the appropriate limits (again denoted by $G_\om$),
that is, taking the closure
of the set $\{(G_\om,S_\om) \mid \om \in \Om_p \setminus \Om_{p,0})\}$ in $\m{M}_3$,
one obtains a compact subset $\m{G}_p=\{(G_\om,S_\om) \mid \om \in \Om_p\}$ of $\m{M}_3$ which is 
homeomorphic to $\Om_p$ (via the correspondence $\om \mapsto (G_\om,S_\om)$)
and hence homeomorphic to  a Cantor set. In what follows we will continue to keep the notation $G_\om,\;\om\in\Om_p$
to denote this groups after this modification. Also quite often we will identify $\m{G}_p$
with $\Om_p$.
In the case $p=2$ the new limit groups $G_\om,\; \om \in \Om_{2,0}$
are known to be virtually metabelian groups of exponential growth and there
is no analogous result for the case $p>2$. This is the underlying reason that
Theorem \ref{main2} below is stated for only $p=2$.

Another important feature  of the construction is  that for all but
countably many $\om\in \Om_2$  and for all $\om \in \Om_p$  the groups $G_{\om}$ and $(G_{\tau(\om)})^p$ (direct
product of $p$ copies of $G_{\om}$) are abstractly commensurable (i.e., contain  isomorphic subgroups of
finite index). Thus shift $\tau$ preserves many of   group properties on the set of full measure when $\mu$ is a
$\tau$ invariant measure supported on $\Om_{p,\infty}$,  for instance,  the property to
be a torsion group.   While   for  some properties of the groups  $G_{\om}$ it is quite easy to decide whether
it is typical or not,  there  are some  properties for which such a question is more
difficult to answer. Among them  is the  property  to have a growth function bounded  from above 
(or below) by a specific function.

 Given  functions $f_1,f_2:\mathbb{N}\to \mathbb{N}$, we write 
$f_1 \preceq f_2$ if $f_1$ grows no faster than $f_2$ and 
$f_1\sim f_2$ if $f_1 \preceq f_2$ and $f_2 \preceq f_1$.  $f_1\prec f_2$ means that 
$f_1\preceq f_2$ but $f_1 \nsim f_2$
   (precise definitions are  given in Section \ref{prelim}).  For any $\om\in\Om_p$ let $\gamma_\om(n)$ denote the growth
function of the group $G_\om$.
It was shown in \cite{grigorch:degrees,grigorch:degrees85} that if symbols of the alphabet  $\{0,1,\dots,
p\}$  are more or less uniformly distributed in a sequence  $\om$  then $\gamma_\om(n)$ grows 
slower than $e^{n^{\alpha}}$  with  constant $\alpha <1$.  At the same time, in the case $p=2$  
for any function $f(n)\prec e^n$ there is a sequence $\om$ such that $\gamma_\om(n)$ grows not slower than $f(n)$.

The upper and lower bounds  by functions of the type $e^{n^{\alpha}}$  with  constant $0<\alpha <1$  are of
special importance   in study of growth of finitely generated groups and there is a number of
interesting results and conjectures associated with  them. 
One of the  main conjectures says that if the growth of a group $G$ is slower than 
$e^{\sqrt{n}}$ then it is actually polynomial \cite{grigorch:ICM90,grigorch:milnor2011,grigorch:gapconj12}.   

We are ready to formulate our
results. 


\begin{thm} 
\label{ana1}
Suppose $\mu$ is a Borel probability measure on $\Om_p$ that is invariant and
ergodic relative to the shift transformation $\tau:\Om_p\to\Om_p$.   
\begin{itemize}
 \item[a)]If the measure $\mu$ is supported on $\Om_{p,\infty}$, 
then there exists $\alpha=\alpha(\mu,p)<1$ such that $\ga_\om(n)
\preceq e^{n^\alpha}$ for $\mu$-almost all $\om\in\Om_p$.
\item[b)]In the case  $\mu$ is the uniform Bernoulli measure on $\Om_2$, one can take $\alpha=0.999$.
\end{itemize}

\end{thm}If $T:\mathcal{G}_p\rightarrow \mathcal{G}_p$ is a map induced by the shift $\tau$,
our result can be interpreted as follows: For any "reasonable" $T$-invariant measure $\mu$
on $\m{G}_p \subset \m{M}_3$, a typical group in $\m{G}_p$ has growth bounded by $e^{n^\alpha}$,
where $\alpha=\alpha(\mu,p)<1$.

  The bound   for $\alpha$ given in part (b) of the Theorem \ref{ana1} is far from to be   optimal,  but  getting an
essentially better bound  would require more work. In any case it can not be below $1/2$ as for all
groups $G_\om$ of intermediate growth the corresponding growth function is bounded from below by $e^{n^{1/2}}$
\cite{grigorch:degrees,grigorch:degrees85,grigorch:hilbert}.  The  gap conjecture, discussed in \cite{grigorch:gapconj12} and
proven in certain cases, gives more information about what one can expect concerning  possible optimal
values of $\alpha$.

In fact, there is nothing special about the space $\m{M}_3$ and the following holds:

\begin{thmoneprime}
 For any $k\geq 2$ and prime $p$, $\mathcal{M}_k$ contains a compact subset 
 $\mathcal{K}_k=\{(M_\om, L_\om) \mid \om \in \Om_p\}$ homeomorphic
 to $\Omega_p$ (via the map $\om \mapsto (M_\om,L_\om)$) such that if $\mu$ is a measure supported on $\Om_{p,\infty}$ 
 there exists $\alpha=\alpha(\mu,p)<1$ such that $\ga_{M_\om}(n)
\preceq e^{n^\alpha}$ for $\mu$-almost all $\om\in\Om_p$. 
\end{thmoneprime}

For $k\geq 3$, the group $M_\om$ is the same as $G_\om$, with an appropriate generating set $L_\om$ of size $k$.  
For $k=2$, $M_\om$ is a 2-generated group constructed from $G_\om$
using an idea from \cite{grigorch:degrees} (see Section \ref{primetheorems}). 

\medskip

The proof of Theorem \ref{ana1} is based on the next result,  which has its own interest.
It improves \cite[Theorem 3]{grigorch:degrees85} (a generalization of 
\cite[Theorem 3.2]{grigorch:degrees}), which states that  
if the sequence $\om$ is regularly packed by symbols $0,1,2$, namely, if there exists $k=k(\om)$
such that each subsequence of length $k$ of $\om$ contains all symbols $\{0,1,\ldots,p\}$,
then  there is  $\alpha <1$ such that $\gamma_\om(n) \preceq e^{n^\alpha}$.

To every
infinite word $\om=l_1l_2\dots$ in $\Om_{p,\infty}$ we associate an increasing
sequence of integers $t_i=t_i(\om)$, $i=0,1,2,\dots$ Namely, $t_i$ is the smallest
integer such that the finite word $l_1l_2\ldots l_{t_i}$ can be split into $i$ subwords each
containing all letters $\{0,1,\ldots,p\}$.
For any $C\geq p+1$ let $\Om_{p,C}$ denote the set of all infinite words
$\om\in\Om_{p,\infty}$ such that $t_n(\om)\le Cn$ for sufficiently large $n$.
Given $\eps>0$, let $\Om_{p,C,\eps}$ denote the set of all $\om\in\Om_{p,C}$ such
that $t_{n+1}(\om)-t_n(\om)\le\eps t_n(\om)$ for sufficiently large $n$.

\begin{thm}
\label{important4}
Given $C\geq p+1$, there exist $\eps>0$ and $0<\alpha<1$ such that $\ga_\om(n )\preceq
e^{n^\alpha}$ for any $\om\in\Om_{p,C,\eps}$.
\end{thm}

%
%
%
%
%

Our next result deals exclusively with the case $p=2$ (the reason was explained earlier).
Given two functions $\gamma_1,\gamma_2:\mathbb{N}\to \mathbb{N}$ such that $\gamma_1(n) \prec \gamma_2(n) \prec e^n$, let us say
that a group $G$ has \emph{oscillating growth of type} $(\gamma_1,\gamma_2)$ if
 $\gamma_1 \npreceq \gamma_G$ and $\gamma_G \npreceq \gamma_2$. The existence of groups with oscillating growth
follows from the results of \cite{grigorch:degrees}. 
The results of \cite{kassabov_pak:11} and of \cite{bart-ersch-permgrow},\cite{brieussel} provide much more
information in this direction.
 
Let $\theta_0=\log(2)/ \log(2/x_0)$, where $x_0$ is the real root of the polynomial
 $x^3+x^2+x-2$. We have $\theta_0 < 0.767429$. 
\begin{thm}\mbox{}
 \label{main2}
\begin{itemize}
 \item[a)]For any $\theta >\theta_0$ and any function $f$ satisfying 
 $e^{n^\theta}\prec f(n) \prec e^n$, there exists 
a dense $G_\delta$ subset $\m{Z} \subset\m{G}_2$
 such that any group in $\m{Z}$ has oscillating growth of 
 type $\left(e^{n^\theta},f\right)$.
 \item[b)] There exists a dense $G_\delta$ subset of $\m{G}_2$ which consists of
 groups with  oscillating growth of type $\left(e^{n^\theta},e^{n^\beta}\right)$ 
 for every  $\theta$ and $\beta$, $\theta_0<\theta<\beta<1$.
\item[c)] Given any $\eps>0$ and  function $f$ satisfying 
$\exp\left(\frac{n}{\log^{1-\epsilon}n}\right) \prec f(n) \prec e^n$, there is a dense $G_\delta$ 
subset $\m{E}\subset \{(G_\om,S_\om) \mid \om \in \{0,1\}^\mathbb{N}\}$
such that any group in $\m{E}$ has oscillating growth of type 
$\left( \exp\left(\frac{n}{\log^{1-\epsilon}n}\right),f \right)$.
\end{itemize}
\end{thm}

Again, all these results  generalize to arbitrary $k\geq 2$. In particular, the following theorem holds.

\begin{thmtwoprime}
\label{main2prime}
 For each $k\geq 2$, $\theta>\theta_0$ and function $f$ satisfying $e^{n^\theta}\prec f(n)\prec e^n$, $\mathcal{M}_k$ contains a compact subset $\mathcal 
 C_k$ homeomorphic
 to $\Omega_2$, such that there is a dense $G_\delta$ subset $\mathcal{C}'_k\subset \mathcal C_k$
 which consists of groups with oscillating growth of type $\left(e^{n^{\theta}},f\right)$.
\end{thmtwoprime}

The reason why oscillating groups are typical in the categorical sense is the existence
of a countable dense subset in $\m{G}_2$ consisting of (virtually metabelian) groups of
exponential growth and a dense subset of groups with the growth equivalent to the growth
of the \emph{first Grigorchuk group} $G_{(012)^\infty}$, which is bounded by $e^{n^{\theta_0}}$
 due to a result of Bartholdi \cite{bartholdi:growth} (see also \cite{muchnik_p:growth}). Note also that this is the smallest upper bound
 known for any group of intermediate growth. It is used to prove part a).
To prove part c), we use instead a result of Erschler \cite{erschler:boundary04} stating that the growth of 
the group $G_\om$ for $\om=(01)^\infty\in \Om_2$ is slower than
$\exp\left(\frac{n}{\log^{1-\epsilon}n}\right)$ for all $\eps>0$.

Notice that the categorical approach  for study of amenability  of groups from the family $\mathcal G_2$  
was suggested by Stepin in \cite{stepin96}  where the fact  that this family contains a
dense set of virtually metabelian (and hence amenable) groups was used to show that amenability is a typical
property of this family. In fact, all groups in $\mathcal G_2$ are amenable as was shown in \cite{grigorch:degrees}, but
Stepin's paper provided for the first time  a categorical approach to study typical groups in compact subsets
of the space of marked groups.

\section{Preliminaries}
\label{prelim}

\subsection*{Definition of the groups }

The original definition of  groups in \cite{grigorch:degrees,grigorch:degrees85} is in terms of measure preserving
transformations of the unit interval.  We will give here the   alternative definition 
in terms of automorphisms of rooted trees. For the sake of notation 
we will focus on the case $p=2$ and the construction in the 
  case $p\geq3$ is analogous.
 For more detailed account of this construction see \cite{grigorchuk_schur, grigorch:solved}.

Let us recall some notation:  $\Om_2$ denotes the set all infinite sequences over 
the alphabet $\{0,1,2\}$. 
We identify $\Om_2$ with
the product $\{0,1,2\}^{\bN}$ and endow it with the product topology. 
Let $\Om_{2,0}$ be the set of eventually
constant sequences and $\Om_{2,\infty}$ be the set of sequences in which each letter 0,1,2 appears infinitely
often. Our
notation here is different from \cite{grigorch:degrees,grigorch:degrees85}. Let $\tau :\Om_2 \rightarrow \Om_2$ denote the shift transformation, that is if $\om=l_1l_2\ldots$ then
$\tau(\om)=l_2l_3\ldots$. Note that both $\Om_{2,0}$ and $\Om_{2,\infty}$ are $\tau$ invariant. 

For each $\om \in \Omega_2$ we will define a subgroup $G_\om$
of $Aut(\mathcal{T}_2)$, where the latter denotes the automorphism group of the binary
rooted tree $\m{T}_2$ whose vertices are identified with the set of finite sequences $\{0,1\}^\ast$.
Each group $G_\om$ is the subgroup generated by the four automorphisms denoted by
$a,b_\om,c_\om,d_\om$
whose actions onto the tree is as follows:

\smallskip

For  $v\in\{0,1\}^*$
$$a(0v)=1v \;\text{and} \; a(1v)=0v  $$

$$\begin{array}{llllll}
 b_\om(0v)=& 0 \beta(\om_1)(v) &   c_\om(0v)=& 0 \zeta(\om_1)(v)
  &  d_\om(0v)= &0 \delta(\om_1)(v) \\

   b_\om(1v)=& 1 b_{\tau (\om)}(v) &      c_\om(1v)= & 1 c_{\tau \om}(v)
   & d_\om(1v)= & 1 d_{\tau \om}(v), \\

\end{array}
$$
where
$$
\begin{array}{ccc}
 \beta(0)=a & \beta(1)=a & \beta(2)=e \\
  \zeta(0)=a & \zeta(1)=e & \zeta(2)=a \\
   \delta(0)=e & \delta(1)=a & \delta(2)=a \\
\end{array}
$$
and $e$ denotes the identity.  From the definition, the following relations are immediate:
\begin{equation}
 a^2=b_\om^2=c_\om^2=d_\om^2=b_\om c_\om d_\om =e.
 \label{basicrels}
\end{equation}
Observe that the group $G_\om$ is in fact 3-generated  as one of generators  $b_\om, c_\om, d_\om$ can
be deleted from the generating set. We will use the notation $A_\om=\{a,b_\om,c_\om,d_\om\}$ 
for this  generating set while  $S_\om$ will denote the reduced generating set $\{a,b_\om,c_\om\}$ 
(as in Section \ref{sec2}).
Algebraically, the action defines an embedding into the semi-direct product

$$
\begin{array}{cccllllr}
 \varphi_\om: & G_\om & \rightarrow &S_2  & \ltimes &  (G_{\tau (\om)}& \times& G_{\tau (\om)} ) \\
& & & & & & & \\
& a & \mapsto &(01)& &(e &,&e) \\
&  b_\om & \mapsto & & &(\beta(\om_1) & , & b_{\tau(\om)})   \\
 &  c_\om & \mapsto & & &(\zeta(\om_1)& , & c_{\tau(\om)})  \\
 &  d_\om & \mapsto & & &(\delta(\om_1)& , & d_{\tau(\om)})  \\
\end{array}
$$
where $S_2$ is the symmetric group of order 2 and $(01)$ denotes its non-identity element.

Given $g \in G_\om $ and $x\in \{0,1\}$ let us denote the $x$ coordinate
of $\varphi_\om(g)$ by $g_x$ (or by $g|_x$ to avoid possible confusion) so that $\varphi_\om(g)=\sigma_g (g_0,g_1)$. Let us also extend this to all $ \{0,1\}^*$ by
$$g_{xv}=(g_x)_v$$ where $x\in \{0,1\}$ and $v \in \{0,1\}^*$. For $g\in G_\om$ and $v\in\{0,1\}^\ast$ 
the automorphism $g_v$ will be called \textit{the section of $g$ at vertex $v$}.
Note that if $v$ has length $n$ and $g\in G_\om$, then $g_v$
is an element of $G_{\tau^n(\om)}$. Given $g,h\in G_\om$ and $v\in \{0,1\}^*$ we have
\begin{equation}
 (gh)_v=g_{h(v)}h_v
 \label{sect}
\end{equation}

\subsection*{Topology in the space of marked groups}

A \textit{marked k-generated group} is a pair $(G,S)$ where $G$ is a
group and $S=\{s_1,\ldots,s_k\}$ is an ordered set of (not necessarily distinct) generators of $G$.
The \textit{canonical map} between two marked k-generated groups $(G,S)$ and $(H,T)$ is the
map that sends $s_i$ to $t_i$ for $i=1,2,\ldots,k$.
Let $\mathcal M_k$ denote the  \textit{space of marked k-generated groups} consisting
of marked k-generated groups where two marked groups are identified whenever the canonical map
between them extends to an isomorphism of the groups.

There is a natural metric on $\mathcal M_k$: Two marked groups $(G,S),(H,T)$ are of distance
$\frac{1}{2^m}$ where $m$ is the largest natural number such that the canonical map between
$(G,S)$ and $(H,T)$ extends to an isomorphism (of labeled graphs) 
from the ball of radius $m$ (around the identity)  in the Cayley graph of $(G,S)$ onto the ball of
radius $m$  in the Cayley graph  of $(H,T)$. 
This makes $M_k$ into a compact, totally disconnected topological space.

Alternatively, let $F_k$ be free over the ordered basis $X=\{x_1,\ldots,x_k\}$ and let $\mathcal N(F_k)$
denote the set of normal subgroups of $F_k$. $\mathcal N(F_k)$ has a natural topology
 inherited from the space $\{0,1\}^{F_k}$ of all subsets of $F_k$.
$\mathcal M_k$ can be identified with $\mathcal N(F_k)$ in the following 
way: Each $(G,S)\in \mathcal M_k$ is identified with the kernel of the canonical map 
between $(F_k,X)$ and $(G,S)$. Conversely each 
$N\vartriangleleft F_k$ is identified with $(F_k / N, \{\bar x_1,\ldots,\bar x_k\})$
where $\{\bar x_1,\ldots,\bar x_k\}$ is the image of the basis of $F_k$ in $F_k / N$.
A system of basic open sets are sets of the form
$\mathcal O_{A,B}=\{N \vartriangleleft F_k \mid A \subset N, B\cap N=\varnothing\}$
where $A$ and $B$ are finite subsets of $F_k$. Or the topology can be defined by the metric 
$d(N_1,N_2)=2^{-m}$ where $m=\max\{n \mid B_{F_k}(n)\cap N_1 = B_{F_k}\cap N_2\}$.
It is easy to see that the topology defined in this way  agree
with the definition given in the previous paragraph
(see \cite{champetier_guirardel:limit}
for a survey of  alternative definitions).

Let $A_\om=\{a,b_\om,c_\om,d_\om\}$ so that  $\mathcal F_2=\{(G_\om,A_\om) \mid \om \in \Omega_2\}$ 
is a subset of $\mathcal M_4$. $\mathcal F_2$ is not closed in $\mathcal M_4$ 
(see \cite{grigorch:degrees}). Given  $\om \in \Om_{2}$, let $\{\om^{(n)}\}\subset 
\Om_{2}\setminus \Om_{2,0}$
be a sequence converging to $\om$. It was shown in \cite{grigorch:degrees} 
that the sequence $\{(G_{\om^{(n)}},A_{\om^{(n)}})\}$ converges in $\m{M}_4$ to a marked group
$(\widetilde G_\om,\widetilde A_\om)$ that depends only on $\om$. Moreover, 
$(\widetilde G_\om,\widetilde A_\om)= (G_\om,A_\om)$ if and only if $\om \in \Om_{2}\setminus \Om_{2,0}$.
By construction, the group $\widetilde G_\om$ acts naturally on the binary rooted tree for any $\om \in \Om_2$. 
However the action is not faithful 
 when $\om \in \Om_{2,0}$. 
The modified
 family $ \{(\widetilde G_\om,\widetilde A_\om) \mid \om \in \Omega_2\}$  is a compact subset of $\mathcal M_4$
homeomorphic to $\Omega_2$ via the map $\widetilde G_\om \mapsto \om$.

Observe that a similar procedure can be applied to the family 
$\m{G}_2=\{(G_\om,S_\om) \mid \om \in \Om_2\}$ to obtain a closed subset $\{(\widetilde G_\om,\widetilde S_\om) \mid \om \in \Om_2\}$
 in $\m{M}_3$.
In what follows we will mostly  be concerned with the modified groups. Therefore, 
we use notation $\m{F}_2$ and $\m{G}_2$ for the modified families and also drop all tildes.

\subsection*{Growth functions of groups} Given a  group $G$ and a
finite generating set $S$ of $G$, the growth function  of $G$ with respect
to $S$  is defined as $\gamma_G^S(n)=|B(n)|$ where $B(n)$ is the ball of radius $n$
around the identity in the Cayley graph of
$G$ with respect to the generating set $S$.

Given two increasing functions $f,g :\bN \rightarrow \bN$, write $f\preceq g$ if there exists
a constant $C>0$ such that $f(n)\leq g(Cn)$ for all $n \in \bN$. Also let
$f \sim g$ mean that $f\preceq g$ and $g \preceq f$ with the  convention that 
$f \prec g$ meaning $f \preceq g$ but $f\nsim g $. It can be easily observed that
$\sim$ is an equivalence relation and the growth functions of a group with respect to different
generating sets are $\sim$ equivalent. Therefore one can speak of the growth of a group
meaning the $\sim$ equivalence class of its growth functions.
Note that if two groups $(G,S),(H,T) \in \mathcal M_k$  are of distance $2^{-m}$ , then
$\gamma_G^S(n)=\gamma_H^T(n)$ for $n\leq m$.

If $G$ is an \emph{infinite} group and $H$ a subgroup of finite index, then the
growth functions of $G$ and $H$ are $\sim$ equivalent 
by \cite[Proposition 3.1]{grigorch:degrees} (note that this is not true if $G$ is a finite group).
Therefore if two finitely generated infinite groups $G_1$ and $G_2$ are commensurable (i.e., have finite
index subgroups $H_1,H_2$ which are isomorphic) then their growth functions are $\sim$ 
equivalent.

There are three types of growth for groups: If $\gamma_G \preceq n^d$ for some $d\geq 0$ then
$G$ is said to be of polynomial growth, if $\gamma_G \sim e^n$ then it is said to have
exponential growth. If neither of this happens then the group is said to have
\textit{intermediate growth}. Also the condition $\gamma_G \prec e^n$ means that $G$
has \emph{subexponential} growth.

\begin{defn}\label{oscillating}

Let $G$ be a finitely generated group with  growth function 
$\gamma_G$ corresponding to some generating set. Let $\gamma_1,\gamma_2$ be two
functions such that $\gamma_1(n) \prec \gamma_2(n) \prec e^n$. $G$ is said to 
have \emph{oscillating growth of type} $(\gamma_1,\gamma_2)$ if $\gamma_1 \npreceq \gamma_G$
 and $\gamma_G \npreceq \gamma_2$ (i.e., neither 
 $\gamma_1 \preceq \gamma_G $ nor $\gamma_G \preceq \gamma_2$).

%
%
%
%
\end{defn}
Equivalently, the group $G$ has oscillating growth of type  $(\gamma_1,\gamma_2)$
if for some (and hence for all ) generating set $S$ we have the following:
For every $C\in \mathbb{N}$ there exists $m=m(C)$ such that $\gamma_G^S(Cm)<\gamma_1(m)$
and for every $D\in \mathbb{N}$ there exists $k=k(D)$ such that $\gamma_2(Dk)<\gamma_G^S(k)$.


 Regarding the growth of the groups 
 $\mathcal G_p$ the following are known (recall that $\gamma_\om(n)$ denotes the growth
 function of $G_\om$ and when $\om \in \Om_{p,0}$, $G_\om$ denotes the limit group 
 obtained by the procedure described above).

\begin{thm}\mbox{}

\begin{enumerate}
 \item If $\om\in \Omega_2 \setminus \Om_{2,0}$ or $\om \in \Om_{p,\infty}$ if $p\geq 3$,  then $G_\om$ is of intermediate growth.
 \item If $\om \in \Om_{2,0}$ then $G_\om$ is of exponential growth.
  \item For every $\om \in \Omega_2$ or $\om \in \Om_{p,\infty},
  \;p\geq 3$, we have  $e^{\sqrt{n}} \preceq \gamma_\om(n)$.

 \item If there exists a number $r$ such that
 every subword of $\om$ of length $r$ contains all the symbols $\{0,1,\ldots,p\}$ then $\gamma_\om(n) \preceq e^{n^\alpha}$ for some
 $0<\alpha<1$ depending only on $r$.

 \item There  is a subset $\Lambda \subset \Omega_2$ of the cardinality of continuum  such that the
 functions $\{\gamma_\om(n) \mid \om \in \Lambda\}$ are incomparable with respect to $\preceq$.
 
 \item For any function $f(n)$ such that $f(n) \prec e^n$, there exists $\omega \in \Omega_2\setminus \Om_{2,0}$
for which $\gamma_\om(n) \npreceq f(n)$.

 \item If $\om=(012)^\infty\in \Om_2$ is the periodic sequence with period $012$ then 
 $e^{n^{\alpha_0}} \preceq \gamma_\om(n) \preceq e^{n^{\theta_0}}$,
 where $\alpha_0=0.5157,\theta_0=\log(2)/ \log(2/x_0)$ and $x_0$ is the real root of the polynomial
 $x^3+x^2+x-2$ ($\theta_0\approx 0.7674$).
 \item If $\om=(01)^\infty\in \Om_2$ is periodic with period $01$ then $\exp\left(\frac{n}{\log^{2+\epsilon}n}\right)\preceq \gamma_\om(n)
\preceq \exp\left(\frac{n}{\log^{1-\epsilon}n}\right)$ for any $\epsilon>0$.
\end{enumerate}
\label{general}

\end{thm}

\begin{proof}\mbox{}
 \begin{enumerate}
  \item See \cite[Theorem 3.1]{grigorch:degrees} and \cite{grigorch:degrees85}.
  \item See \cite[Lemma 6.1]{grigorch:degrees}.
  \item  See \cite[Theorem 3.2]{grigorch:degrees} and \cite[Theorem 4.4]{grigorch:degrees85} 
  where the lower bound $e^{\sqrt{n}}$ is proven for a certain subset of $\Om_2$ and for 
  $\Om_{p,\infty},\;p\geq 3$. As all groups mentioned are residually finite $p$-groups
  for some prime $p$ and are not virtually nilpotent (which can be shown in various ways, for 
  example using the fact that the groups are periodic),
  the lower bound $e^{\sqrt{n}}$ follows from a general result of
  \cite{grigorch:hilbert}.
  \item See \cite{grigorch:degrees} and  \cite{bartholdi_s:growth}, \cite{muchnik_p:growth} for explicit upper bounds
  depending on $r$.
  \item See \cite[Theorem 7.2]{grigorch:degrees}.
  \item See \cite[Theorem 7.1]{grigorch:degrees}.
  \item  See \cite{bartholdi:lower} for the lower bound 
   which improved upon \cite{leonov:bound}. See  \cite{bartholdi:growth} for the
      upper bound.
  \item See \cite{erschler:boundary04}.
 \end{enumerate}

\end{proof}

\section{Proof of  Theorem \ref{ana1}}

This section is devoted to the proof of Theorems \ref{ana1} and \ref{important4}.
We prove these theorems in the case $p=2$. The proof in the case $p\geq 3$ is completely analogous.
To simplify notation, we set $\Om=\Om_2$ and $\Om_\infty=\Om_{2,\infty}$ for the rest of this section.

For any element $g$ of a group
$G_\om$, $\om\in\Om$, we denote by $|g|$ its length relative to the
generating set $A_\om=\{a,b_\om,c_\om,d_\om\}$.  If $|g|=n$, then $g$ can
be expanded into a product $s_1s_2\dots s_n$, called a geodesic representation, 
where each $s_i\in A_\om$.
For every generator $s\in A_\om$ we denote by $|g|_s$ the number of times
this generator occurs in the sequence $s_1,s_2,\dots,s_n$.  
Note that the element $g$ may admit several geodesic representations and $|g|_s$
may depend on a representation (for example, $b_\om a d_\om a b_\om = c_\om a d_\om a c_\om$
for any $\om$ starting with $0$). Lemmas \ref{l-0} and \ref{l-3} below hold for any 
possible value of the corresponding number $|g|_s$.

\begin{lem}\label{l-0}

$(|g|-1)/2\le|g|_a\le(|g|+1)/2$ for all $g\in G_\om$.

\end{lem}

\begin{proof}
It follows from relations (\ref{basicrels}) that
 any geodesic representation of an element 
 $g\in G_\om$ is  of the form
 $$g=(s_1)as_2a\ldots a(s_k),$$ where each $s_i\in \{b_\om,c_\om,d_\om\}$ and  parentheses indicate 
optional factors. The lemma follows.

\end{proof}

\begin{lem}\label{l-1}

For any word $w\in\{0,1\}^*$ of length $q$ we have  $|g_w|\le 2^{-q}|g|+1-2^{-q}$.

\end{lem}

\begin{proof} First consider the case when $w$ is $0$ or $1$. 
 Let $g=s_1s_2\dots s_n$ be a geodesic representation, where each $s_i\in A_\om$.
It follows by induction from equation (\ref{sect}) that $g_w=s_1|_{w_1}\,s_2|_{w_2}\ldots s_n|_{w_n}$,
where $w_n=w$ and $w_i=(s_{i+1}\ldots s_n) (w)$, for $1 \le i \le n-1$.
Notice that each section $s_i|_{w_i}$ is a generator of the group $G_{\tau(\om)}$ or $e$. Moreover,
$s_i|_{w_i}=e$ if $s_i=a$. Therefore $|g_w|\le |g|-|g|_a$. By Lemma \ref{l-0}, $|g|_a\ge (|g|-1)/2$.
Hence $|g_w| \le (|g|+1)/2$. Equivalently, $|g_w|-1\le 2^{-1}(|g|-1)$.

Now it follows by induction on $|w|$ that $|g_w|-1\le 2^{-q}(|g|-1)$ for any word $w$ of length $q$.

%
\end{proof}

For any element $g\in G_\om$ and any integer $q\ge0$, let

$$L_q(g)=\sum_{|w|=q}|g_w|$$
\begin{lem}\label{l-2}

$L_q(gh)\le L_q(g)+L_q(h)$ for all $g,h\in G_\om$.

\end{lem}

\begin{proof}
Since $(gh)_w=g_{h(w)}h_w$ for any word $w\in \{0,1\}^\ast$, it follows that $|(gh)_w|\le |g_{h(w)}|+|h_w|$.
Summing this inequality over all words $w$ of length $q$ and using the fact that $h$ acts bijectively on such
words, we obtain $L_q(gh)\le L_q(g)+L_q(h)$.

\end{proof}

\begin{lem}\label{l-3}

$L_q(g)\le |g|+1-|g|_{h_q}$ for any $q\ge1$, where $h_q=b_\om$, $c_\om$, or
$d_\om$ if the $q$-th letter of $\om$ is $2$, $1$, or $0$, respectively.

\end{lem}

\begin{proof}

Let $n=|g|$.  Consider an arbitrary geodesic representation $g=s_1s_2\dots s_n$, where each $s_i\in A_\om$.  It
follows by induction from Lemma \ref{l-2} that $L_q(g)\le L_q(s_1)+L_q(s_2)
+\dots+L_q(s_n)$.  Fix an arbitrary word $w\in\{0,1\}^*$ of length
$q$.  Clearly, $a|_w=1$.  Further, $h_q|_w=1$ unless $w=1\ldots 1$ (in which case
$h_q|_w\in A_{\tau^q\om}\setminus\{a\}$).  If $s$ is any of the other two generators in
$A_\om$, then $s|_w=1$ unless $w=1\ldots 1$ (in which case $s|_w\in
A_{\tau^q\om}\setminus \{a\}$) or $w=1\ldots 10$ (in which case $s|_w=a$).  Therefore
$L_q(a)=0$, $L_q(h_q)=1$, and $L_q(s)=2$ if $s\in A_\om$ is neither $a$ nor
$h_q$.  It follows that $L_q(g)\le 2(|g|-|g|_a)-|g|_{h_q}$.  By Lemma
\ref{l-0}, $|g|_a\ge(|g|-1)/2$.  Hence $2(|g|-|g|_a)\le|g|+1$.

\end{proof}

\begin{lem}\label{l-4}

Suppose that the beginning of length $q$ of the sequence $\om$ contains 
each of the letters $0$, $1$, and $2$.  Then

$$L_q(g)\le \frac56|g|+\frac76+2^{q-1}$$
for all $g\in G_\om$.

\end{lem}

\begin{proof}
We have $|g|=|g|_a+|g|_{b_\om}+|g|_{c_\om}+|g|_{d_\om}$ whenever all numbers in the right-hand side
are computed for the same geodesic representation of $g$.
By Lemma \ref{l-0}, $|g|_a\le(|g|+1)/2$.  It follows that $|g|_s\ge (|g|-1)/6$ for some
$s\in\{b_\om,c_\om,d_\om\}$.  Lemma \ref{l-3} implies that $L_{q_0}(g)\le
|g|+1-|g|_s\le \frac56|g|+\frac76$ for some $1\le q_0\le q$.  In the case
$q_0=q$, we are done.  Otherwise we notice that $L_q(g)=\sum_{|w|=q_0}
L_{q-q_0}(g_w)$.  By Lemma \ref{l-3}, $L_{q-q_0}(g_w)\le |g_w|+1$ for any
word $w$.  Therefore $L_q(g)\le L_{q_0}(g)+2^{q_0}\le \frac56|g|+\frac76
+2^{q-1}$.

\end{proof}

%
%

Note that the growth function of a group is \textit{sub-multiplicative}, that is,
$\gamma(n+m)\leq\gamma(n)\gamma(m)$ for every $n,m \in \bN$.
It is convenient to extend the argument of a growth function to non-integer values:
Given increasing $f  :\bN \rightarrow \bN$ , define $\tilde f :\mathbb{R}^+ \rightarrow \bN$
 by $\tilde f(x)=f
 (\lceil x\rceil)$ for
all  $x$ where $\lceil x\rceil$ is the least natural number bigger than or equal to $x$.
Observe that $f(x+\kappa)\le\tilde f(x)$ whenever
$\kappa<1$. If $f:\bN \rightarrow \bN$ is sub-multiplicative then it is easy to see that
$\tilde f(x+y)\le\tilde f(x)\tilde f(y)$ for any $x,y>0$.

For the remainder of this section let $\rho=\frac{131}{132}$.

 \begin{lem}\label{l-6}

Suppose that the beginning of length $q$ of the sequence $\om$ features
each of the letters $0$, $1$, and $2$.  Then
$$\tilde\ga_\om(x)\le 2^{2^{q+1}}\Biggl(\tilde\ga_{\tau^q\om}
\left(\frac{x}{11\cdot2^q}\right)\Biggr)^{\rho(11\cdot2^q)}$$
for any $x>0$.

\end{lem}

\begin{proof}

Let $n=\lceil x\rceil$ and consider an arbitrary element $g\in G_\om$ of
length at most $n$.  By Lemma \ref{l-1}, we have $|g_w|\le 2^{-q}n+1-2^{-q}$
for any word $w\in\{0,1\}^*$ of length $q$.  We denote by $W$ the
set of all words $w$ of length $q$ such that $|g_w|>\frac{12}{11}
\cdot2^{-q}(\frac56n+\frac76+2^{q-1})$.  In view of Lemma \ref{l-4}, the
cardinality of $W$ satisfies $|W|<\frac{11}{12}\cdot2^q$.

The element $g$ is uniquely determined by its sections on words of length
$q$ and its restriction to the $q$th level of the binary rooted tree.  The
number of possible choices for the restriction is at most $2^{2^q}$.  The
number of possible choices for the set $W$ is also at most $2^{2^q}$.  Once
the set $W$ is specified, the number of possible choices for a particular
section $g_w$ is at most $\ga_{\tau^q\om}(\frac{12}{11}
\cdot2^{-q}(\frac56n+\frac76+2^{q-1}))$ if $w\notin W$ and at most
$\ga_{\tau^q\om}(2^{-q}n+1-2^{-q})$ otherwise.  Since $n<x+1$, we have
$2^{-q}n+1-2^{-q}<2^{-q}x+1$ so that $\ga_{\tau^q\om}(2^{-q}n+1-2^{-q})\le
\tilde\ga_{\tau^q\om}(x/2^q)$.  Besides, $\frac{12}{11}
\cdot2^{-q}(\frac56n+\frac76+2^{q-1})<\frac{10}{11}2^{-q}x+1$ so that
$\ga_{\tau^q\om}(\frac{12}{11}\cdot2^{-q}(\frac56n+\frac76+2^{q-1}))\le
\tilde\ga_{\tau^q\om}(\frac{10}{11}x/2^q)$.  Finally, for a fixed set $W$
the number of possible choices for all sections of $g$ is

$$\displaystyle
\begin{array}{cl}
\tilde\ga_{\tau^q\om}\left(\frac{10x}{11\cdot 2^q}\right)^{2^q-|W|}
\tilde\ga_{\tau^q\om}\left(\frac{x}{2^q}\right)^{|W|} & \le
\tilde\ga_{\tau^q\om}\left(\frac{x}{11\cdot 2^q}\right)^{10(2^q-|W|)+11|W|} 
\le\tilde\ga_{\tau^q\om}\left(\frac{x}{11\cdot2^q}\right)^{\frac{131}{132}\cdot11\cdot2^q}.
\end{array}
$$
Consequently,

\begin{equation}
\label{iterate}
 \tilde\ga_\om(x)\le 2^{2^{q+1}}\tilde\ga_{\tau^q(\om)}
\left(\frac{x}{11\cdot2^q}\right)^{\rho(11\cdot2^q)}.
\end{equation}

\end{proof}

As in Section \ref{sec2}, to every
infinite word $\om=l_1l_2\dots$ in $\Om_{\infty}$ we associate an increasing
sequence of integers $t_i=t_i(\om)$, $i=0,1,2,\dots$.  The sequence is
defined inductively.  First we let $t_0=0$.  Then, once some $t_i$ is
defined, we let $t_{i+1}$ to be the smallest integer such that the finite
word $l_{t_i+1}l_{t_i+2}\dots l_{t_{i+1}}$ features each of the letters
$0$, $1$, and $2$. Further, let $q_i=t_i-t_{i-1}$ for $i=1,2,\ldots$

\begin{lem}\label{dohuz}
Let $x_m=11^m\cdot2^{t_m}$ for any integer $m>0$.  Then $\ga_\om(x_m)\le
10^{\rho^mx_m}$.
\end{lem}

\begin{proof}
For any integer $m>0$ let $\al_m=\rho(11\cdot 2^{q_m})$ and
$\be_m=2^{q_m+1}$.  Lemma \ref{l-6} implies that
$$
\tilde\ga_{\tau^{t_{m-1}}(\om)}(x)\le 2^{\be_m}\,
\tilde\ga_{\tau^{t_m}(\om)}\left(\frac{x}{11\cdot2^{q_m}}\right)^{\al_m}
$$
for any $x>0$.  Since $q_1+q_2+\dots+q_m=t_m$, it follows that for any
integer $m>0$ and real $x>0$,
$$
\tilde\ga_\om(x)\le 2^{S_m}\,\tilde\ga_{\tau^{t_m}(\om)}
\left(\frac{x}{11^m\cdot2^{t_m}}\right)^{R_m},
$$
where $R_m=\al_1\dots\al_m$ and $S_m=\be_1+\al_1\be_2+\dots
+\al_1\dots\al_{m-1}\be_m$.  In particular, $\ga_\om(x_m)\le 2^{S_m}
\ga_{\tau^{t_m}(\om)}(1)^{R_m}=2^{S_m}5^{R_m}$.  Since $R_m=\rho^m x_m$, it
remains to show that $S_m\le R_m$.

We have $\al_m=\frac{11}{2}\rho\be_m=\frac{131}{24}\be_m>5\be_m$.  Notice
that $q_m\ge3$ so that $\be_m\ge16$.  Hence $\al_m-\be_m>64$.  Now the
inequality $S_m\le R_m$ is proved by induction on $m$.  First of all,
$S_1=\be_1<\al_1=R_1$.  Then, assuming $S_m\le R_m$ for some $m>0$, we get
$S_{m+1}=S_m+\al_1\dots\al_m\be_{m+1}\le R_m+\al_1\dots\al_m\be_{m+1}
=\al_1\dots\al_m(1+\be_{m+1})<R_{m+1}$.
\end{proof}


Recall some notation from Section \ref{sec2} (for brevity, we drop index $p$).
For any $C\geq0$ let $\Om_{C}$ denote the set of all infinite words
$\om\in\Om_\infty$ such that $t_n(\om)\le Cn$ for sufficiently large $n$.
Given $\eps>0$, let $\Om_{C,\eps}$ denote the set of all $\om\in\Om_{C}$ such
that $q_{n+1}=t_{n+1}(\om)-t_n(\om)\le\eps t_n(\om)$ for sufficiently large $n$.


Now we can prove the next theorem, which is  a more detailed version of Theorem \ref{important4} (in the case $p=2$).

\begin{thm}
Let $C>0$ and $\alpha > 1-\dfrac{\log(\rho^{-1})}{\log(11\cdot2^{C})}$.
Then there exists $\eps>0$ such that $\ga_\om(n)\preceq
e^{n^{\alpha}}$ for any $\om\in\Om_{C,\eps}$.
\label{tek}
\end{thm}

\begin{proof} 
Let $\kappa=\dfrac{\log(\rho^{-1})}{\log(11\cdot2^{C})}$.
Note that $0<\kappa<1$.
Choose $\epsilon>0$  small enough so that $(\eps+1)(1-\kappa)<\alpha$.
Let $\om \in \Om_{C,\eps}$. Then there exists an integer $N>0$ such that  
$ t_m\leq Cm$ and  $q_{m+1}=t_{m+1}-t_m\leq \epsilon t_m $
 for $m\ge N$.
By the choice of $\kappa$ we have

\begin{equation*}
 \rho^m=\left(\frac{1}{(11\cdot 2^C)^{\kappa}}\right)^m= \frac{1}{(11^m\cdot 2^{Cm})^{\kappa}}
\leq\frac{1}{(11^m\cdot 2^{t_m})^{\kappa}}=x_m^{-{\kappa}}
\end{equation*}
for any $m\ge N$.
Since $$x_{m+1}=11^{m+1}\cdot 2^{t_{m+1}}=11\cdot2^{q_{m+1}}\cdot11^m\cdot2^{t_m}=11\cdot
2^{q_{m+1}}\cdot x_m,$$
we obtain
$$
\begin{array}{cl}
 x_{m+1}^{1-\kappa}=(11\cdot 2^{q_{m+1}}\cdot x_m)^{1-\kappa}
& \leq 11^{1-\kappa}  (11^{\epsilon m}\cdot2^{\epsilon t_{m}}\cdot x_m)^{1-\kappa} \\
 &  =  11^{1-\kappa}x_m^{(\epsilon+1)(1-\kappa)}  \le11^{1-\kappa}x_m^{\alpha}.
\end{array}
$$

Consider an arbitrary integer $n\ge x_N$. We have  $x_m\leq n \leq x_{m+1}$
for some $m\ge N$. By Lemma \ref{dohuz},
$$\gamma_{\om}(n)\leq \gamma_\om(x_{m+1})\leq 10^{\rho^{m+1}x_{m+1}}.$$
By the above,
$$\rho^{m+1}x_{m+1}\leq x_{m+1}^{1-\kappa}\leq 11^{1-\kappa}x_m^{\alpha}\leq
11^{1-\kappa}n^{\alpha},$$
hence
$$\gamma_\om(n)\leq 10^{11^{1-\kappa}n^{\alpha}}=D^{n^{\alpha}}, $$
where $D=10^{11^{1-\kappa}}$. Thus $\gamma_\om(n)
\preceq D^{n^\alpha} \sim e^{n^\alpha}$.

\end{proof}


Suppose $\mu$ is a Borel probability measure on $\Om$ that is invariant and
ergodic relative to the shift transformation $\tau:\Om\to\Om$.  Since
$\Om_{\infty}$ is a Borel, shift invariant set, the measure $\mu$ is either
supported on $\Om_{\infty}$ or else $\mu(\Om_{\infty})=0$.
Theorem \ref{ana1} will be derived from Theorem \ref{tek} using the following lemma.

\begin{lem}\label{l-th}

If the measure $\mu$ is supported on $\Om_{\infty}$, then there exists
$C_0>0$ such that $t_n(\om)/n\to C_0$ as $n\to\infty$ for $\mu$-almost all
$\om\in\Om$.  Consequently, $\mu(\Om_{C,\eps})=1$ for any $C>C_0$ and
$\eps>0$. In the case $\mu$ is the uniform Bernoulli measure on $\Om$, we can take $C_0<7.3$.

\end{lem}

\begin{proof}

For any finite word $w$ over the alphabet $\{0,1,2\}$ let $T(w)$ denote the
maximal number of non overlapping sub-words of $w$ each containing all the
letters.  Clearly, $T(w)\le |w|/3$.  It is easy to see that $T(w_1)+T(w_2)
\le T(w_1w_2)\le T(w_1)+T(w_2)+1$ for any words $w_1$ and $w_2$.  It
follows by induction that $T(w_0)+T(w_1)+\dots+T(w_k)\le T(w_0w_1\dots w_k)
\le T(w_0)+T(w_1)+\dots+T(w_k)+k$ for any words $w_0,w_1,\dots,w_k$.

For any $\om\in\Om$ and integer $m>0$ let $T_m(\om)=T(\om_m)$, where
$\om_m$ is the beginning of length $m$ of the sequence $\om$.  We are going
to show that for $\mu$-almost all $\om$ there is a limit of $T_m(\om)/m$ as
$m\to\infty$.  Note that $T_m$ is a bounded ($0\le T_m\le m/3$) Borel
function on $\Om$.  Let

$$I_m=\int_{\Om} T_m\,d\mu.$$
If $\om\in\Om_{\infty}$ then $T_m(\om)>0$ for $m$ large enough.  Since the
measure $\mu$ is supported on the set $\Om_{\infty}$, it follows that $I_m>0$
for $m$ large enough.

Given integers $m_1,m_2>0$, the beginning of length $m_1+m_2$ of any
sequence $\om\in\Om$ is represented as the concatenation of two words, the
beginning of length $m_1$ of the same sequence and the beginning of length
$m_2$ of the sequence $\tau^{m_1}(\om)$.  Therefore $T_{m_1+m_2}(\om)\ge
T_{m_1}(\om)+T_{m_2}(\tau^{m_1}(\om))$.  Integrating this inequality over
$\Om$ and using shift-invariance of the measure $\mu$, we obtain
$I_{m_1+m_2}\ge I_{m_1}+I_{m_2}$.  Now the standard argument implies that
$I_m/m\to I$ as $m\to\infty$, where $I=\sup_{k\ge1}I_k/k$.  Note that
$0<I\le1/3$.

Let $\Om_\mu$ denote the Borel set of all sequences $\om\in\Om$ such that for any
integer $m>0$ we have
$$\lim_{k\to\infty} \frac1k\sum\nolimits_{i=0}^{k-1} T_m(\tau^i(\om))=I_m.$$
Birkhoff's ergodic Theorem implies that $\Om_\mu$ is a set of full measure:
$\mu(\Om_\mu)=1$.
Consider an arbitrary $\om\in\Om$ and integers $m>0$ and $k\ge2m$.  Let
$l=\lfloor k/m\rfloor$, the integer part of $k/m$.  For any integer $j$, $0\le j<m$, we represent the
beginning of length $k$ of $\om$ as the concatenation of $l$ words
$w_0w_1\dots w_l$, where $w_0$ is of length $j$, $w_l$ is of length
$k-lm+m-j$, and the other words are of length $m$.  By the above,

$$T_k(\om)-l\le\sum\nolimits_{i=0}^l T(w_i)\le T_k(\om).$$

By construction, $T(w_i)=T_m(\tau^{(i-1)m+j}(\om))$ for $1\le i\le l-1$.
Besides, $l\le k/m$ and $0\le T(w_0)+T(w_l)\le (k-lm+m)/3<2m/3$.  Therefore

$$T_k(\om)-k/m-2m/3\le\sum\nolimits_{i=1}^{l-1} T_m(\tau^{(i-1)m+j}(\om))\le
T_k(\om).$$
Summing the latter inequalities over $j$ ranging from $0$ to $m-1$, we
obtain

$$mT_k(\om)-k-2m^2/3\le\sum\nolimits_{i=0}^{(l-1)m-1} T_m(\tau^i(\om))\le mT_k(\om).
$$
Since $0\le\sum_{i=(l-1)m}^{k-1}T_m(\tau^i(\om))\le (k-lm+m)m/3<2m^2/3$, it
follows that

$$mT_k(\om)-k-2m^2/3\le\sum\nolimits_{i=0}^{k-1} T_m(\tau^i(\om))\le mT_k(\om)+2m^2/3.$$
Then

$$
\biggl|\frac1k T_k(\om)-\frac1{mk}\sum\nolimits_{i=0}^{k-1}T_m(\tau^i(\om))
\biggr|\le \frac1m+\frac{2m}{3k}.$$

At this point, let us assume that $\om\in\Om_\mu$.  Fixing $m$ and
letting $k$ go to infinity in the latter estimate, we obtain that all limit
points of the sequence $\{T_k(\om)/k\}_{k\ge1}$ lie in the interval
$[I_m/m-1/m,I_m/m+1/m]$.  Letting $m$ go to infinity as well, we obtain
that $T_k(\om)/k\to I$ as $k\to\infty$.

Given $\om\in\Om_{\infty}$, there is a simple relation between sequences
$\{T_m(\om)\}_{m\ge1}$ and $\{t_n(\om)\}_{n\ge1}$.  Namely, $t_n(\om)\le m$
if and only if $T_m(\om)\ge n$.  In particular, $T_{t_n(\om)}(\om)=n$.
Since $T_m(\om)/m\to I$ as $m\to\infty$ for any $\om\in\Om_\mu$, it
easily follows that $t_n(\om)/n\to C_0$, where $C_0=I^{-1}$, for any
$\om$ in $\Om_\mu\cap\Om_{\infty}$, a set of full measure.

Now consider the case $\mu$ is the uniform Bernoulli measure.  To estimate
the limit $C_0$ in this case, we are going to evaluate the integral $I_7$.
For any integer $k\ge0$ let $N_k$ denote the number of words $w$ of length
$7$ over the alphabet $\{0,1,2\}$ such that $T(w)=k$.  Then
$I_7=3^{-7}\sum_{k\ge0}kN_k$.  Since $N_k=0$ for $k>2$, we have
$N_0+N_1+N_2=3^7$ and $I_7=(N_1+2N_2)/3^7$.  Let us compute the numbers
$N_0$ and $N_2$.  A word $w$ of length $7$ satisfies $T(w)=0$ if it does
not use one of the letters.  The number of words missing one particular
letter is $2^7$.  Also, there are three words $0000000$, $1111111$, and
$2222222$ that miss two letters.  It follows that $N_0=3\cdot2^7-3=381$.
To compute $N_2$, we represent an arbitrary word $w$ of length $7$ as
$w_1lw_2$, where $w_1$ and $w_2$ are words of length $3$ and $l$ is a
letter.  There are two cases when $T(w)=2$.  In the first case, each of the
words $w_1$ and $w_2$ contains all letters, then $l$ can be arbitrary.  In
the second case, either $w_1$ or $w_2$ misses exactly one letter, then $l$
must be the missing letter and the other word must contain all letters.  It
follows that $N_2=(3!)^2\cdot 3+2M\cdot 3!$, where $M$ is the number of
words of length $3$ that miss exactly one of the letters $0$, $1$, and $2$.
It is easy to observe that $M=3^3-3!-3=18$, then $N_2=324$.  Now
$N_1=3^7-N_0-N_2=2187-381-324=1482$.  Finally, $I_7=(N_1+2N_2)/3^7
=(1482+2\cdot324)/3^7=710/729$.  Now we can estimate the limits.
As shown earlier, $I\ge I_7/7=710/(7\cdot729)>100/729$, then
$C_0=I^{-1}<7.3$.

\end{proof}

Now we are ready to complete the proof of Theorem \ref{ana1}.

 Take any $C>C_0$, where $C_0$ is as in Lemma \ref{l-th}. By Theorem
\ref{tek},  there exists $\epsilon >0$ and $0<\alpha<1$ such that $\gamma_\om(n) \preceq e^{n^{\alpha}}$
for all $\om \in \Om_{C,\eps}$. The set $\Om_{C,\eps}$ has full measure by Lemma \ref{l-th}.
In the case when $\mu$ is the uniform Bernoulli measure, we can assume that $C<7.3$ by Lemma \ref{l-th}.
Consequently, we can choose  $\alpha= 1-\kappa$, where $\kappa=\log\frac{132}{131}/\log(11\cdot2^{7.3})$.  
One can compute that $\kappa>0.001$.
\begin{flushright}
 $\square$
\end{flushright}

%
%
%
%
%


\section{Proof of Theorem \ref{main2}} 

 Recall that we are in the case $p=2$ so that we use the notation
 $\Om=\Om_2$ and $\Om_0=\Om_{2,0}$. Also, recall $\theta_0$ is as defined before Theorem \ref{main2}.
We begin with  preliminary lemmas.

\begin{lem}
\label{lower}

Let $g$ be a function of natural argument and let $\m{L}_g\subset \m{M}_k$ be the subset  consisting of marked groups
$(G,S)$ such that $g\npreceq \gamma_G^S$. Then $\m{L}_g$ is a $G_\delta$ subset of $\m{M}_k$ (i.e., a countable 
intersection of open sets).

 
\end{lem}

\begin{proof}

Given $(G,S)\in \m{L}_g$ and $C \in \mathbb{N}$, let $K=K((G,S),C)$ be
such that $\gamma_G^S(CK)<g(K)$ (such $K$ exists since $g \npreceq \gamma_G^S$).
 Let $\m{B}_{((G,S),C)}$ denote the ball of radius $2^{-CK}$ 
 (in the metric defined in the space of marked groups) centered at $(G,S)$. 
  We claim that 
  $$\m{L}_g = \bigcap_{C\in \mathbb{N}} \bigcup_{(G,S)\in \m{L}_g}\m{B}_{((G,S),C)}  .$$ 
The inclusion $\subset$ is clear. For the other inclusion, let $(H,T)$ be an element of the right hand side.
Then for any $C\in \mathbb{N}$ there is $(G,S)\in \m{L}_g$ such that 
 $(H,T)\in \m{B}_{((G,S),C)} $. Therefore for $K=K((G,S),C)$ we have $\gamma_H^T(CK)=\gamma_G^S(CK)<g(K)$
 and hence $g \npreceq \gamma_H^T$, which shows that $(H,T)\in \m{L}_g$.


\end{proof}

\begin{lem}
\label{upper}

Let $f$ be a function of natural argument and let $\m{U}_f\subset \m{M}_k$ be the subset consisting of marked groups
$(G,S)$ such that $\gamma_G^S \npreceq f$. Then $\m{U}_f$ is a $G_\delta$ subset of $\m{M}_k$.

 
\end{lem}

\begin{proof}
The proof is analogous to the proof of Lemma \ref{lower}.

%


\end{proof}

Now we are going to prove each part of Theorem \ref{main2}.

\medskip
 
 \textbf{Part a)} Suppose we are given $\theta>\theta_0$ and a function 
 $f(n)\prec e^n$.
 Let $\eta=(012)^\infty\in \Om$ and recall that we have $\gamma_\eta(n) \preceq e^{n^{\theta_0}}$
 (Theorem \ref{general}, part 7).
 Consider the set  $\m{X}=\{(G_\om,S_\om) \mid \tau^k(\om)=\eta \;\; \text{for some} \; k \}\subset \m{G}_2$.
 Since $\m{G}_2$ is homeomorphic to $\Om$ via
 $(G_\om,S_\om)\mapsto \om$, the set $\m{X}$  is dense in $\m{G}_2$. For any
 $\om \in \Om \setminus \Om_0$, the groups $G_\om$ and $G_{\tau(\om)}\times G_{\tau(\om)}$ are commensurable  
by \cite[Theorem 2.2]{grigorch:degrees}. Therefore  we have
 $$\gamma_\om \sim \gamma_{\tau(\om)}^2 $$
 and for any $\om$, $\tau^k(\om)=\eta$ it follows that
 $$\gamma_\om\sim \gamma_\eta^{2^k} \preceq (e^{n^{\theta_0}})^{2^k}\sim
  e^{n^{\theta_0}}. $$
  Let $g(n)=e^{n^\theta}$ so that $\m{X} \subset \m{L}_g$, where $\m{L}_g$ is defined in Lemma \ref{lower}.
  
  According to Theorem \ref{general}, part 2,
the set  $\m{Y}=\{(G_\om,S_\om) \mid  \om \in \Om_0\}$, which is dense in $\m{G}_2$, 
consists of groups of exponential growth. In particular, $\m{Y}\subset \m{U}_f$, where
$\m{U}_f\subset \m{M}_3$ is defined in Lemma \ref{upper}. 
  By Lemmas \ref{lower} and \ref{upper}, the sets $\m{L}_g \cap \m{G}_2$ and $\m{U}_f\cap \m{G}_2$
  are dense $G_\delta$ subsets of $\m{G}_2$. Since $\m{G}_2$ is compact, their intersection is also 
a dense $G_\delta$ subset of $\m{G}_2$. For any $(G,S)\in \m{L}_g\cap\m{U}_f\cap\m{G}_2$, we have 
$g \npreceq \gamma_G^S$ and $\gamma_G^S \npreceq f$.
  
  \smallskip 
  
  \textbf{Part b)} This part is a corollary of part a) with $f(n)=e^{\frac{n}{\log n}}$ as $e^{n^\beta} \prec f(n) $
for any $\beta<1$.
  
  \smallskip
  
  \textbf{Part c)} The proof of this part is analogous to part a).
  Let us denote by $\Om'=\{0,1\}^\mathbb{N}\subset \Om$.
  We  set $\zeta=(01)^\infty$, 
  $\m{X}=\{(G_\om,S_\om) \mid \tau^k(\om)=\zeta \;\; \text{for some} \; k \}$
  and $\m{Y}=\{(G_\om,S_\om) \mid \om \in \Om_0\cap\Om'\}$. According to  Theorem \ref{general}, part 8,
  for any $\eps>0$ the function $g(n)=\exp\left(\frac{n}{\log^{1-\epsilon}n}\right)$ grows faster than
$\gamma_\zeta$. Since $g\sim g^2$, it follows that $\gamma_\om \preceq g$ whenever $(G_\om,S_\om) \in \m{X}$.
 It remains to apply Lemmas \ref{lower} and \ref{upper}.

\section{Proof of theorems \ref{ana1}' and 3' }
\label{primetheorems}

As it was mentioned in the introduction there is a natural embedding 
$\iota_k : \m{M}_k \rightarrow \m{M}_{k+1}$ given by $\iota_k((G,A))=(G,A')$
where $A'=\{a_1,\ldots,a_k,a_{k+1}\}$ if $A=\{a_1,\ldots,a_k\}$
and $a_{k+1}=1$ in $G$. This induces an embedding $\iota_{k,n}:\m{M}_k \rightarrow \m{M}_{k+n}$
for all $k,n$ and given a subset $X\subset \m{M}_k$ one can consider its homeomorphic
image $\iota_{k,n}(X)\subset \m{M}_{k+n}$.

There are two natural ways of replacing one generating set $A$ of a group $G$ by another.
The one, as was just suggested, by adding one more formal generator representing the identity
(and placing it for definiteness at the end), or applying to a generating set  
\emph{Nielsen transformations} which are given by (see \cite{mks}):

\begin{itemize}
\item[i)] Exchanging two generators ,
 \item[ii)] Replacing a generator $a\in A$ by its inverse $a^{-1}$,
 \item[iii)] Replacing $a_i\in A$ by $a_ia_j$ where $a_i\neq a_j$.

\end{itemize}
Note that these transform generating sets into generating sets.
It is in general incorrect that two generating sets of size $k$ of a group are related by a 
sequence of Nielsen transformations (i.e. by an automorphism of the free group $F_k)$, but
we have the following:

\begin{prop}
 Let $(G,A)\in\m{M}_k$ and $(G,B)\in \m{M}_n$. Let $\iota_{k,n}(G,A)=(G,A')$ and $\iota_{n,k}(G,B)=(G,B')$.
  so that $(G,A'),(G,B')\in \m{M}_{n+k}$. Then $A'$ can be transformed into $B'$ by a sequence
  of Nielsen transformations.
\end{prop}

\begin{proof}
 Let $A'=\{a_1,\ldots,a_k,a_{k+1},\ldots,a_{k+n}\}$ and
   $B'=\{b_1,\ldots,b_n,b_{n+1},\ldots,b_{n+k}\}$ where 
  $a_{k+1}=\ldots =a_{k+n}=b_{n+1},\ldots =a_{n+k}=1$.
  For all $b_i,\; 1\le i \leq n$ there is a word $B_i \in \{a_1,\ldots,a_k\}^\pm$ such that
  $b_i=B_i$. By a sequence of Nielsen transformations (acting trivially on $a_i,\; i \leq k $) we can transform $A'$
  into $A''=\{a_1,\ldots,a_k,B_1,\ldots,B_n\}=\{a_1,\ldots,a_k,b_1,\ldots,b_n\}$. 
  In a similar way $B'$ can be transformed into $B''=\{b_1,\ldots,b_n,a_1,\ldots,a_k\}$
  by a sequence of Nielsen transformations.
  It is clear that $B''$ can be obtained from $A''$ by permuting the generators, which 
  can be achieved by a sequence of Nielsen transformations.
\end{proof}

Taking the inductive limit $\m{M}=\lim_{\rightarrow} \m{M}_k$ and setting 
$A_\infty=\{a_1,a_2,\ldots\}$ the previous proposition shows 
(as observed by Champetier in \cite{champ:grps_fini}) that
the group of Nielsen transformations over an infinite alphabet  
(that is, the group $Aut_{fin}(F_\infty)$ of finitary automorphisms
of a free group $F_\infty$ of countably infinite rank) acts on $\m{M}$
in such a way that if two pairs $(G,A),(G,B)\in \m{M}$ represent the same group
then they belong to the same orbit of the action of $Aut_{fin}(F_\infty)$ on $\m{M}$
(and it is clear the points in the orbit all represent the same group). In \cite{champ:grps_fini}
it was shown that this action, which is by homeomorphisms and hence is Borel, is not
\emph{tame} (in other terminology not measurable, or not smooth). As was mentioned
in the introduction, the question of existence of $Aut_{fin}(F_\infty)$-invariant 
(or at least quasi-invariant) measure is important for the topic of random groups.

There are more general ways of embedding $\m{M}_k$ into $\m{M}_l$. Assume we have 
a subset $X=\{(G_i,A_i) \mid i \in I \} \subset \m{M}_k$ where 
$A_i=\{a^{(i)}_1,\ldots,a^{(i)}_k\}$. Let $F_k$ be a free group  on $\{a_1,\ldots,a_k\}$
and suppose that there are words $B_j(a_1,\ldots,a_k) \in F_k$ for $1\leq j\leq m$
such that for all $i\in I$, the set 
$$B_i=\{B_1(a^{(i)}_1,\ldots,a^{(i)}_k),\ldots,B_m(a^{(i)}_1,\ldots,a^{(i)}_k)  \}$$
is a generating set for $G_i$. Let $Y=\{(G_i,B_i) \mid i \in I \}\subset \m{M}_m$.

\begin{prop}
 The map $\varphi : X \rightarrow Y$ given by $\varphi((G_i,A_i))=(G_i,B_i)$
 is a homeomorphism.
 \label{prop2}
\end{prop}

\begin{proof}
 Let $X'=\iota_{k+m}(X)$ and $Y'=\iota_{m+k}(Y)$. By the previous proposition, there is an automorphism
 of $F_{k+m}$ (realized by a sequence of Nielsen transformations) which 
 induces a homeomorphism $\phi$ of $\m{M}_{k+m}$  which maps $X'$ onto  $Y'$. It is 
 clear that $\varphi$  is the restriction of $\iota_{m+k}^{-1}\circ \phi \circ \iota_{k+m}$
 to $X$.

\end{proof}

We are ready to prove the theorems.

Theorems \ref{ana1} and \ref{main2}
show that theorems 1' and 3' hold for $k=4$. Using the previous propositions
it immediately follows that they hold for values $k\geq 4$.

For $k=3$ observe that by virtue of equations (\ref{basicrels}),
for every $\om \in \Om$  we have $d_\om=b_\om c_\om$ and hence the groups $G_\omega$
are generated by $\{a,b_\om,c_\om\}$. Therefore by proposition \ref{prop2} we
obtain the result for $k=3$.

The case $k=2$ is more delicate. There are several methods of embedding a group into
a $2$-generated group.  We need an embedding that preserves the property
to have intermediate growth. 
To accomplish this we use an idea from \cite{grigorch:degrees}. Let $T$ be the rooted tree with branch index 4,2,2,2,\ldots. Given $\om \in \Om$,
let $e$ be the automorphism of $T$ which cyclically permutes the first level
vertices and let $f_\om$ be the automorphism given by $(b_\om,c_\om,a,1)$. 
Set $M_\om=\langle x,y_\om \rangle$. This gives an embedding 

$$
\begin{array}{cccccc}
 \psi: M_\om  & \rightarrow &  S_4 & \ltimes & G_\om^4 \\
 x & \mapsto & \sigma  & &(1,1,1,1) \\
 y_\om & \mapsto &  & &(b_\om,c_\om,a,1)  
\end{array}
$$
where $\sigma$ is the cyclic permutation of order 4 in $S_4$. 
Let $$\bar M_\om=\langle y_\om, xy_\om x^{-1},x^2y_\om x^{-2},x^3y_\om x^{-3} \rangle$$
 and observe that $\bar M_\om$ has index 4 in $M_\om$. The equalities
 
 $$
 \begin{array}{ccc}
  \psi(y_\om)  & = & (b_\om,c_\om,a,1)  \\
  \psi(x y_\om  x^{-1}) &  = &  (c_\om,a,1,b_\om) \\
  \psi(x^2 y_\om x^{-2}) &  = & (a,1,b_\om,c_\om) \\
  \psi(x^3 y_\om x^{-3}) & =& (1,b_\om,c_\om,a)
 \end{array}
 $$
 show that $\bar M_\om$ is a sub-direct product of $G_\om^4$. 
 Hence if $G_\om$ has intermediate growth so does $M_\om$ and if the growth
 of $G_\om$ is bounded above by a function of the form $e^{n^\alpha}$
 then the same holds for the growth function of $M_\om$. Similarly, if $G_\om$ has 
oscillating growth of type $(e^{n^\theta},f)$, so does $M_\om$.
 
 One can observe that the branch algorithm solving the word problem for groups 
 $G_\om$ (described in \cite{grigorch:degrees}) can be adapted 
 to the groups $M_\om$: The covering group will be 
 $\mathbb{Z}_4 \ast \mathbb{Z}_2$, given a word $g$ in the normal form in  
 $\mathbb{Z}_4 \ast \mathbb{Z}_2$ first one checks whether the exponent of $e$
 in $g$ is divisible by 4 or not. If not then the element $g$ does not belong
 to the first level stabilizer and hence $g\neq 1$. Otherwise
 one computes the sections of $g$ and then applies the classical branch algorithm
 to the sections of $g$ with oracle $\om$.
  This shows that  for two sequences $\om,\eta \in \Om \setminus \Om_0$ which have
 common prefix of length $n$, the Cayley graphs of the  groups $M_\om$ and $M_\eta$ 
 will have isomorphic balls of radius $2^{n-1}$. 
 Therefore we consider the subset $X=\{(M_\om,L_\om) \mid \om \in \Om \setminus \Om_0\} \subset \m{M}_2$
 where $L_\om=\{x,y_\om\}$ and take its closure in $\m{M}_2$ to obtain a Cantor set in $\m{M}_2$.
 The new limit groups $ M_\om$ for $\om \in \Om_0$ will have a finite index subgroup
 which is a sub-direct product in the group $ G_\om^4$, and therefore are of exponential
 growth. Thus the limit groups $ M_\om,\;\om \in \Om_0$ will have exponential
 growth and therefore similar arguments used to prove Theorem   \ref{main2} can be applied 
 in this case too. Also note that when $\om \in \Om_{\infty}$ then $M_\om$ and 
 $G^4_\om$ are abstractly commensurable i.e. have finite index subgroups which are isomorphic.

 For $p\geq 3$ a similar construction can be done by setting $M_\om=\langle x,y_\om \rangle$
 as the group of automorphisms of the tree with branch index $p^2,p,p,\ldots$,
 where $x$ is the cyclic permutation of order $p$ and $y_\om=(b_\om,c_\om,a,1,\ldots,1)$. 
 One can observe that  $M_\om$ in this case is a sub-direct product in $G_\om^{p^2}$
 and $M_\om$ is abstractly commensurable with $G_{\tau(\om)}^{p^2}$ when $\om \in \Om_{p,\infty}$.
 This allows to prove the theorem in the case $p\geq3$.

\section{Concluding Remarks}

Let $G_\om^{um},\;\om \in \Om_{p,0}$ denote the unmodified groups as 
defined in Section 3 (i.e., the groups before modifying countably many groups corresponding
to eventually constant sequences). Note that for fixed prime $p$ we have
$G^{um}_{0^\infty}=G^{um}_{1^\infty}=\ldots=G^{um}_{p^\infty}$ as subgroups of the $p$-ary
rooted tree.
The limit groups $G_\om,\; \om \in \Om_{p,0}$ map onto the corresponding group $G_\om^{um}$.
When $p=2$ and $\om \in \Om_2$ is a constant sequence then $G_\om^{um}$
is isomorphic to the infinite dihedral group \cite[Lemma 2.1]{grigorch:degrees} and hence 
has linear growth. This shows that $G^{um}_\om$ has polynomial growth for $\om \in \Om_{2,0}$
For $p\geq 3$ and $\om \in \Om_p $  a constant sequence, 
the groups $G_\om^{um}$ were considered in \cite{grigorch:jibranch}
and were shown to be regular branch self-similar groups. As these groups are
residually finite $p$-groups, the main result of \cite{grigorch:hilbert} shows that  for all such groups
$e^{\sqrt{n}}$ is a lower bound for their growth functions. Therefore for all primes
$p$ and $\om \in \Om_{p,0}$, the groups $G_\om$ have super-polynomial growth.
As mentioned in Theorem \ref{general}, for $p=2$ the groups $G_\om\;\om \in \Om_{2,0}$
are known to have  exponential growth. An extension of this fact to $p>2$ would generalize
Theorem \ref{main2} to all primes $p$. For $p=3$ and $\om\in\Om_p$
a constant sequence, the group
$G^{um}_\om$ coincides with the Fabrykowski-Gupta group studied in \cite{fg91}.
In \cite{bartholdi-pochon} it was shown that the growth of this group satisfies
$$ e^{n^{\frac{\log3}{\log6}}} \preceq \gamma(n) \preceq e^{{\frac{n(\log\log n)^2}{\log n}}}.$$

%
%

A more general problem is the following:
Given two increasing functions $\gamma_1,\gamma_2$ such that 
$\gamma_i(n)\sim\gamma_i(n)^p,\;i=1,2$ and $\gamma_1(n) \prec \gamma_2(n)$, consider the set
$$\mathcal{W}_{\gamma_1,\gamma_2}=\{\om \in \Om_p \mid \gamma_1(n) 
\preceq \gamma_\om(n) \preceq \gamma_2(n)\} .$$
As mentioned before, for $\om \in\Om_p \setminus \Om_{p,0}$ the groups $G_\om$ and $G_{\tau(\om)}$
are commensurable and hence $\gamma_\om \sim \gamma^p_{\tau(\om)}$. This shows that 
$\mathcal{W}_{\gamma_1,\gamma_2}$ is $\tau$ invariant, and hence for any $\tau$ invariant 
ergodic measure $\mu$ defined on $\Om_p$ we have $\mu(\mathcal{W}_{\gamma_1,\gamma_2})=0\;\text{or}\;1$.
A natural direction for investigation would be to determine  functions $\gamma_1,\gamma_2$
for which the set $\mathcal{W}_{\gamma_1,\gamma_2}$ has full measure (and make $\gamma_1,\gamma_2$
as close to each other as possible while keeping $\mu(\mathcal{W}_{\gamma_1,\gamma_2})=1$).
Theorem \ref{ana1}, part (b) together with Theorem \ref{general}, part (3) can be interpreted as  
$\mu(\m{W}_{\gamma_1,\gamma_2})=1$,
where $\gamma_1(n)=e^{n^{0.5}},\gamma_2(n)=e^{n^{0.999}}$ and $\mu$
is the uniform Bernoulli measure on $\Om_2$.

The idea  of  statements  similar  to  Lemmas  \ref{lower}  and  \ref{upper}, which descends  to the paper
of   A. Stepin \cite{stepin96}, is based on the fact that many group properties are formulated in 
"local  terms" with respect to the topology on $\m{M}_k$.
This  includes     properties  such as to  be  amenable, to be LEK (locally embeddable  into  
 the class $K$  of groups), to be  sofic, to be  hyperfinite,  etc. (see, e.g., \cite{tullio:book}).

In  all  these and  other  cases   one  can state  that for  any $k$ the  subset 
$X_{\m{P}}\subset \m{M}_k$ of groups  satisfying   
a local property $\mathcal{P}$   is a $G_\delta$ set in $\m{M}_k$.  So   if a   subset  $Y \subset \m{M}_k$  has a dense subset of groups 
satisfying property $\m{P}$ then it  contains dense  $G_\delta$ subset satisfying  property  $\m{P}$. 
In  some  cases  like  
LEF (locally embeddable into finite groups), LEA (locally  embeddable into amenable  groups), sofic and hyperfinite
groups, 
the corresponding  set is  a  closed  subset  in $\m{M}_k$ and  there  is not  a big outcome of the above  argument.  
But  for   properties such  as to  be amenable, to  have particular  type of growth and  some other  
properties, the above  observation gives a nontrivial  information  about  the structure  of  the  space  of  groups.

\bibliographystyle{alpha}
\bibliography{mylib}

\def\cprime{$'$} \def\cprime{$'$} \def\cprime{$'$} \def\cprime{$'$}
  \def\cprime{$'$} \def\cprime{$'$} \def\cprime{$'$} \def\cprime{$'$}
  \def\cprime{$'$}
\begin{thebibliography}{MKS66}

\bibitem[Bar98]{bartholdi:growth}
Laurent Bartholdi.
\newblock The growth of {G}rigorchuk's torsion group.
\newblock {\em Internat. Math. Res. Notices}, (20):1049--1054, 1998.

\bibitem[Bar01]{bartholdi:lower}
Laurent Bartholdi.
\newblock Lower bounds on the growth of a group acting on the binary rooted
  tree.
\newblock {\em Internat. J. Algebra Comput.}, 11(1):73--88, 2001.

\bibitem[BE12]{bart-ersch-permgrow}
Laurent Bartholdi and Anna Erschler.
\newblock Growth of permutational extensions.
\newblock {\em Invent. Math.}, 189(2):431--455, 2012.

\bibitem[BGH12]{BGH12}
Mustafa~G. Benli, Rostislav Grigorchuk, and Pierre De~La Harpe.
\newblock Amenable groups without finitely presented amenable covers, 2012.
\newblock To appear in the Bulletin of Mathematical Sciences.

\bibitem[BP92]{benedetti:book}
Riccardo Benedetti and Carlo Petronio.
\newblock {\em Lectures on hyperbolic geometry}.
\newblock Universitext. Springer-Verlag, Berlin, 1992.

\bibitem[BP09]{bartholdi-pochon}
Laurent Bartholdi and Floriane Pochon.
\newblock On growth and torsion of groups.
\newblock {\em Groups Geom. Dyn.}, 3(4):525--539, 2009.

\bibitem[Bri11]{brieussel}
J\'er\'emie Brieussel.
\newblock Growth behaviors in the range $e^{r^?}$, 2011.
\newblock (available at \emph{http://arxiv.org/abs/1107.1632}).

\bibitem[B{\v{S}}01]{bartholdi_s:growth}
Laurent Bartholdi and Zoran {\v{S}}uni{\'k}.
\newblock On the word and period growth of some groups of tree automorphisms.
\newblock {\em Comm. Algebra}, 29(11):4923--4964, 2001.

\bibitem[CG05]{champetier_guirardel:limit}
Christophe Champetier and Vincent Guirardel.
\newblock Limit groups as limits of free groups.
\newblock {\em Israel J. Math.}, 146:1--75, 2005.

\bibitem[Cha00]{champ:grps_fini}
Christophe Champetier.
\newblock L'espace des groupes de type fini.
\newblock {\em Topology}, 39(4):657--680, 2000.

\bibitem[CSC10]{tullio:book}
Tullio Ceccherini-Silberstein and Michel Coornaert.
\newblock {\em Cellular automata and groups}.
\newblock Springer Monographs in Mathematics. Springer-Verlag, Berlin, 2010.

\bibitem[Ers04]{erschler:boundary04}
Anna Erschler.
\newblock Boundary behavior for groups of subexponential growth.
\newblock {\em Annals of Math.}, 160(3):1183--1210, 2004.

\bibitem[FG91]{fg91}
Jacek Fabrykowski and Narain Gupta.
\newblock On groups with sub-exponential growth functions. {II}.
\newblock {\em J. Indian Math. Soc. (N.S.)}, 56(1-4):217--228, 1991.

\bibitem[Gri84]{grigorch:degrees}
R.~I. Grigorchuk.
\newblock Degrees of growth of finitely generated groups and the theory of
  invariant means.
\newblock {\em Izv. Akad. Nauk SSSR Ser. Mat.}, 48(5):939--985, 1984.

\bibitem[Gri85a]{grigorch:degrees85}
R.~I. Grigorchuk.
\newblock Degrees of growth of {$p$}-groups and torsion-free groups.
\newblock {\em Mat. Sb. (N.S.)}, 126(168)(2):194--214, 286, 1985.

\bibitem[Gri85b]{grigorch:habil}
R.I. Grigorchuk.
\newblock {\em Groups with intermediate growth function and their
  applications}.
\newblock Habilitation, Steklov Institute of Mathematics, 1985.

\bibitem[Gri89]{grigorch:hilbert}
R.~I. Grigorchuk.
\newblock On the {H}ilbert-{P}oincar\'e series of graded algebras that are
  associated with groups.
\newblock {\em Mat. Sb.}, 180(2):207--225, 304, 1989.

\bibitem[Gri91]{grigorch:ICM90}
Rostislav~I. Grigorchuk.
\newblock On growth in group theory.
\newblock In {\em Proceedings of the {I}nternational {C}ongress of
  {M}athematicians, {V}ol.\ {I}, {II} ({K}yoto, 1990)}, pages 325--338, Tokyo,
  1991. Math. Soc. Japan.

\bibitem[Gri99]{grigorchuk_schur}
R.~I. Grigorchuk.
\newblock On the system of defining relations and the {S}chur multiplier of
  periodic groups generated by finite automata.
\newblock In {\em Groups {S}t. {A}ndrews 1997 in {B}ath, {I}}, volume 260 of
  {\em London Math. Soc. Lecture Note Ser.}, pages 290--317. Cambridge Univ.
  Press, Cambridge, 1999.

\bibitem[Gri00]{grigorch:jibranch}
R.~I. Grigorchuk.
\newblock Just infinite branch groups.
\newblock In {\em New horizons in pro-$p$ groups}, volume 184 of {\em Progr.
  Math.}, pages 121--179. Birkh\"auser Boston, Boston, MA, 2000.

\bibitem[Gri05]{grigorch:solved}
Rostislav Grigorchuk.
\newblock Solved and unsolved problems around one group.
\newblock In {\em Infinite groups: geometric, combinatorial and dynamical
  aspects}, volume 248 of {\em Progr. Math.}, pages 117--218. Birkh\"auser,
  Basel, 2005.

\bibitem[Gri11]{grigorch:milnor2011}
R.~I. Grigorchuk.
\newblock Milnor's problem on the growth of groups and its consequences, 2011.
\newblock (available at \emph{http://arxiv.org/abs/1111.0512}).

\bibitem[Gri12]{grigorch:gapconj12}
R.~I. Grigorchuk.
\newblock On gap conjecture concerning group growth, 2012.
\newblock (available at \emph{http://arxiv.org/abs/1202.6044}).

\bibitem[Gro81]{gromov:poly_growth}
Mikhael Gromov.
\newblock Groups of polynomial growth and expanding maps.
\newblock {\em Inst. Hautes \'Etudes Sci. Publ. Math.}, (53):53--73, 1981.

\bibitem[G{\v S}07]{grigorchuk-s:standrews}
Rostislav Grigorchuk and Zoran {\v S}uni{\'c}.
\newblock Self-similarity and branching in group theory.
\newblock In {\em Groups St. Andrews 2005, I}, volume 339 of {\em London Math.
  Soc. Lecture Note Ser.}, pages 36--95. Cambridge Univ. Press, Cambridge,
  2007.

\bibitem[KP11]{kassabov_pak:11}
M.~Kassabov and I.~Pak.
\newblock Groups of oscillating intermediate growth, 2011.
\newblock (available at \emph{http://arxiv.org/abs/1108.0262}).

\bibitem[Leo01]{leonov:bound}
Yu.~G. Leonov.
\newblock On a lower bound for the growth of a 3-generator 2-group.
\newblock {\em Mat. Sb.}, 192(11):77--92, 2001.

\bibitem[Mil68]{milnor:note68}
J.~Milnor.
\newblock A note on curvature and fundamental group.
\newblock {\em J. Differential Geometry}, 2:1--7, 1968.

\bibitem[MKS66]{mks}
Wilhelm Magnus, Abraham Karrass, and Donald Solitar.
\newblock {\em Combinatorial group theory: {P}resentations of groups in terms
  of generators and relations}.
\newblock Interscience Publishers [John Wiley \& Sons, Inc.], New
  York-London-Sydney, 1966.

\bibitem[MP01]{muchnik_p:growth}
Roman Muchnik and Igor Pak.
\newblock On growth of {G}rigorchuk groups.
\newblock {\em Internat. J. Algebra Comput.}, 11(1):1--17, 2001.

\bibitem[Nek07]{nekrash:minimal07}
Volodymyr Nekrashevych.
\newblock A minimal {C}antor set in the space of 3-generated groups.
\newblock {\em Geom. Dedicata}, 124:153--190, 2007.

\bibitem[Oll05]{Oliv}
Yann Ollivier.
\newblock {\em A {J}anuary 2005 invitation to random groups}, volume~10 of {\em
  Ensaios Matem\'aticos [Mathematical Surveys]}.
\newblock Sociedade Brasileira de Matem\'atica, Rio de Janeiro, 2005.

\bibitem[Ste96]{stepin96}
A.~M. Stepin.
\newblock Approximation of groups and group actions, the {C}ayley topology.
\newblock In {\em Ergodic theory of {${\bf Z}^d$} actions ({W}arwick,
  1993--1994)}, volume 228 of {\em London Math. Soc. Lecture Note Ser.}, pages
  475--484. Cambridge Univ. Press, Cambridge, 1996.

\end{thebibliography}

\end{document}